\documentclass[reqno, 12pt]{amsart}
\usepackage{amsmath,amssymb}
\usepackage{graphicx}
\usepackage{verbatim}
\usepackage[colorlinks,linkcolor=blue,anchorcolor=red,citecolor=red]{hyperref}
\usepackage{enumerate}
\usepackage{upgreek}

\newtheorem{lm}{Lemma}[section]

\newtheorem{thm}{Theorem}[section]
\newtheorem{rmk}{Remark}[section]
\newtheorem{prop}{{\bf Proposition}}[section]

\DeclareMathOperator{\Id}{Id}

\newcounter{saveeqn}%

\makeatletter
\evensidemargin
\oddsidemargin
\makeatother

\title[H\"older regularity of  Hartman-Grobman theorem ]{
H\"older regularity of Hartman-Grobman theorem on vector bundle
}

\author{Xiao Tang}
\address[Xiao Tang]
{School of Mathematical Sciences\\
    Chongqing Normal University\\
    Chongqing 401331, China}
\email[(Corresponding author)]{mathtx@163.com}

\date{}
\allowdisplaybreaks[4]

\begin{document}
\maketitle

\begin{abstract}
In this paper, we study Hölder regularity of the Hartman–Grobman theorem  on vector bundles, considering regularity both along the fiber direction and along the base direction. While the main argument follows standard approaches, additional effort is devoted to constructing a continuous bundle norm that is equivalent to the original norm and with respect to which the linear part becomes hyperbolic at the first iterate.
For the local version of the theorem, we further need to construct a suitable cut-off function on the vector bundle, which is achieved by employing partition of unity.


\vskip 0.2cm

\vskip 0.2cm
{\bf AMS (2020) subject classification:}
37C15; 37D20
\end{abstract}


\baselineskip 15pt   
\parskip 10pt         



\section{Introduction}
\setcounter{equation}{0}
\setcounter{lm}{0}
\setcounter{thm}{0}
\setcounter{rmk}{0}
\setcounter{df}{0}
\setcounter{cor}{0}

The conjugacy problem is a fundamental topic in dynamical systems, as conjugate systems possess the same dynamical properties of interest. Therefore, one of the central goals in the study of dynamical systems is to determine whether a complicated system can be transformed, via a conjugacy, into a simpler one. Among various types of conjugacies, linearization is particularly important, as it concerns the problem of determining whether a nonlinear system can be conjugated to a linear system. This problem has attracted considerable attention from mathematicians. In the analytic setting, Poincar\'e (\cite{Po90}), Siegel (\cite{Siegel54}), and Brjuno (\cite{Brjuno7172}) established fundamental results on analytic linearization. Later, Sternberg (\cite{Sternberg55,Sternberg57}) developed the theory of smooth local linearization. These linearization results require a nonresonance condition. In the framework of random dynamical systems, Li and Lu obtained random analogues of the Poincar\'e theorem (\cite{LL05E}), the Siegel theorem (\cite{LL08}), and the Sternberg theorem (\cite{LL05C}), extending these classical linearization results to random settings.

Regarding topological conjugacy, Hartman (\cite{H60}) and Grobman (\cite{G59}) established, without imposing any nonresonance condition, that a \(C^1\) local hyperbolic diffeomorphism is topologically conjugate to its linear part in a neighborhood of a fixed point. More precisely, let \(F:\mathbb{R}^d\to\mathbb{R}^d\) be a \(C^1\) local diffeomorphism near \(0\), satisfying \(F(0)=0\), and let \(DF(0)=:A\) be hyperbolic, meaning that the eigenvalues of \(A\) do not lie on the unit circle. Then there exists a local homeomorphism \(h\) near \(0\) such that
\begin{equation}\label{loc.conju}
h\circ F = A\circ h.
\end{equation}
In \cite{BR11}, Belitskii and Rayskin further showed that the conjugacy in \eqref{loc.conju} can be chosen to be H\"older continuous in Banach spaces. Subsequently, Barreira and Valls (\cite{BV06}) extended the Hartman-Grobman theorem to nonuniformly hyperbolic systems, and the H\"older regularity of the corresponding conjugacy was established by Backes and Dragi\v{c}evi\'c (\cite{BD22}).
For random dynamical systems, Zhao and Shen (\cite{ZS20}) obtained a random version of the Hartman--Grobman theorem, proving the existence of a random topological conjugacy between a random diffeomorphism and its associated linear cocycle. Moreover, Barreira and Valls (\cite{BV06P}) proved that such a random topological conjugacy possess H\"older regularity. In the setting of vector bundles, the conjugacy equation \eqref{loc.conju} was solved in \cite{BK94}, where the conjugacy is only a homeomorphism.

In this paper, we prove H\"older regularity of the Hartman-Grobman theorem of the setting of a vector bundle. It is established that the conjugacy are both H\"older continuous along fiber and base direction. Now we are going to introduce the setup and state our results.

Let $X$ be a compact metric space, $E$ a continuous vector bundle over $X$, and $\pi: E\to X$ a continuous projection from $E$ to $X$. Assume that $E$ admits a continuous Riemannian metric. 

Let 
$A: E \to E$
be a continuous bundle automorphism covering a homeomorphism \( f: X \to X \), meaning that the following diagram commutes:
\[
\begin{array}{ccc}
E & \xrightarrow{A} & E \\
\downarrow & & \downarrow \\
X & \xrightarrow{f} & X.
\end{array}
\]
We say that \( A \) is a \emph{hyperbolic automorphism} if there exists a continuous splitting
\[
E_x = E^s_x \oplus E^u_x \quad \text{for all } x \in X
\]
such that:
\begin{itemize}
  \item The splitting is \textbf{invariant}, i.e.,
  \[
  A(E^s_x) = E^s_{f(x)}, \quad A(E^u_x) = E^u_{f(x)};
  \]
  \item There exist constants \( C > 0 \) and \( \lambda \in (0, 1) \)  such that for all \( n \geq 0 \):
  \[
  \|A^n(v)\|_{f^n(x)} \leq C \lambda^n \|v\|_x \quad \text{for } v \in E^s_x,
  \]
  \[
  \|A^{-n}(v)\|_{f^{-n}(x)} \leq C \lambda^n \|v\|_x \quad \text{for } v \in E^u_x,
  \]
\end{itemize}
where $\| \cdot\|_x$ denotes the norm on the fiber $E_x$ over $x\in X$, induced by the continuous Riemannian metric. 
This means that vectors in the stable subbundle \( E^s :=\bigsqcup_{x\in M} E^s_x\) contract under forward iterations of \( A \), and vectors in the unstable subbundle \( E^u :=\bigsqcup_{x\in M} E^u_x\) contract under backward iterations. 

In section 2, we will show that there exists a new continuous norm (there is a continuous function $\|\cdot\|^*:E\to 
\mathbb R$ and restricted to each fiber it is a norm), which is uniformly equivalent to the original one such that there exists a constant $\tau\in (0,1)$ satisfying
\begin{align}
\| A(v)\|_{f(x)}^*&\le \tau \|v\|_x^* \quad \text{for } v\in E_x^s,    \label{contrac.}
\\
\| A^{-1}(v)\|_{f^{-1}(x)}^*&\le \tau \|v\|_x^* \quad \text{for } v\in E_x^u,\label{expans.}
\end{align}

Remark that since $A: E\to E$ is continuous and $X$ is compact, $\sup_{x\in X} \| A_x\|$ is bounded above and away from zero. Namely, there exist constants $\ell$ and $L$ such that 
\[ 0< \ell \le \sup_{x\in X} \| A_x\| \le L<\infty,\]
where $\|A_x\|$ stands for the operator norm calculated under the continuous norm $\|\cdot\|^*$.

Let $\phi: E\to E $ be a continuous map such that $\phi_x: E_x \to E_{f(x)}$ is 
 Lipschitz with Lip$(\phi_x) \le  \delta< \min\{ 1-\tau,\ell\}$ for all $x\in X$, $A_x+\phi_x: E_x \to E_x$ is a homeomorphism 
and 
\begin{equation}\label{bound-perturb.}
  B_\phi:=\sup_{x\in X} \sup_{v\in E_x}\|\phi_x(v)\|<\infty.  
\end{equation} 
Define a set $\mathcal{C}^{0,1}_b(E;\delta)$ to be a collection of $\phi$ satisfying the above conditions.

Now we state our results.

\begin{thm}\label{thm1}
  Assume that $A:E\to E$ is  a hyperbolic bundle automorphism covering $f:X\to X$.  Let $\phi: E\to E$ belong to the set  $\mathcal{C}^{0,1}_b(E;\delta)$.  Then there exists a homeomorphism ${\mathcal H}: E\to E$ such that
  \begin{equation}\label{conj.equa.}
 {\mathcal H}_{f(x)} \circ (A_x+\phi_x) = A_x \circ {\mathcal H}_x \qquad \text{for all} \quad x\in X.   
  \end{equation}  
  Moreover, ${\mathcal H}_x$ is chosen to be $\Id + h_x$ for all $x\in X$, where $h_x : E_x \to E_x$ is  bounded and H\"older continuous with exponent 
  \begin{equation}\label{expo.alp.}
\alpha\in \Big(0, \min \Big\{ \frac{\log(\tau+\delta)}{\log(\ell-\delta)}, \frac{\log \{(1-\delta)/\tau\}}{\log(L+\delta)} \Big\}\Big).      
  \end{equation}
  The inverse of $\mathcal H_x$ is also of the form $\Id + \tilde h_x$ for all $x\in X$, where $\tilde h_x : E_x \to E_x$ is also  bounded and H\"older continuous with the same exponent $\alpha$.
\end{thm}

We also want to study H\"older regularity along base direction, which is given as follows.
\begin{thm}\label{thm2}
Under the conditions of Theorem \ref{thm1}, we further assume that 
\begin{align}
\|\hat\Phi_{f(x)} \circ A_x \circ \Phi_x^{-1} - \hat \Phi_{f(\tilde x)} \circ A_{\tilde x} \circ \Phi_{\tilde x}^{-1} \| &\le M_A d(x,\tilde x)^\beta,  
\label{Holder-A}
\\
\|\hat\Phi_{f(x)} \circ \phi_x \circ \Phi_x^{-1}(\tilde v) - \hat \Phi_{f(\tilde x)} \circ \phi_{\tilde x} \circ \Phi_{\tilde x}^{-1}(\tilde v) \| &\le M_\phi d(x,\tilde x)^\beta\| \tilde v\|, 
\label{Holder-phi}
\end{align}
for all $x\in X$ and all $(\tilde x,\tilde v), (x,\tilde v)\in U\times \mathbb R^d$,
where $\Phi: \pi^{-1}(U)\to  U\times \mathbb R^d$ is a local trivialization near $x$,  $\hat \Phi: \pi^{-1}(V)\to  V\times \mathbb R^d$ is a local trivialization near $f(x)$, and $M_A>0, M_\phi>0$ and $0<\beta <1$ are constants independent of $x$, $\tilde x$ and $\tilde v$. We also assume that $f:X\to  X$ is a Lipschitz homeomorphism and $B_\phi$ defined by
\eqref{bound-perturb.} is sufficiently small. 

Then the functions $h$ and $\tilde h$ in Theorem \ref{thm1} both are H\"older continuous along base space $X$ with some exponent $0<\tilde \alpha \le \alpha\beta$. Namely, for each $x\in X$ and all  $(\tilde x,\tilde v), (x,\tilde v)\in U\times \mathbb R^d$, there exists a constant $\tilde M>0$ such that
\begin{align*}
    \| \Phi_x \circ h_x \circ \Phi_x^{-1}(\tilde v) - \Phi_{\tilde x} \circ h_{\tilde x} \circ \Phi_{\tilde x}^{-1}(\tilde v)\| &\le \tilde M d(x,\tilde x)^{\tilde \alpha} \| \tilde v\|^\alpha,
    \\
     \| \Phi_x \circ \tilde h_x \circ \Phi_x^{-1}(\tilde v) - \Phi_{\tilde x} \circ  \tilde h_{\tilde x} \circ \Phi_{\tilde x}^{-1}(\tilde v)\| &\le \tilde M d(x,\tilde x)^{\tilde \alpha} \| \tilde v\|^\alpha.
\end{align*}

\end{thm}

\begin{rmk}\label{rmk1}
 We emphasize that each fiber of the local product $U\times \mathbb R^d$ for local trivialization $\Phi$ near $x$ takes the norm on fiber $E_x$. In other words, for product $U\times \mathbb R^d$, each fiber takes the norm:
\[ \|\tilde v\|_U:=\|\Phi_x^{-1}(\tilde v)\|_x \quad \text{for all } \quad (\tilde x,\tilde v)\in U\times \mathbb R^d. \]
Similarly, each fiber of $V\times \mathbb R^d$ is endowed with the norm:
\[ \|\tilde v\|_V:=\|\hat \Phi_{f(x)}^{-1}(\tilde v)\|_{f(x)} \quad \text{for all} \quad (f(\tilde x),\tilde v)\in V\times \mathbb R^d. \]
Since the norm on vector bundle $E$ is continuous, we conclude that there exists a sufficiently small $\epsilon>0$ such that
\begin{align*}
    (1-\epsilon)\|\Phi_{\tilde x}^{-1}(\tilde v)\|_{\tilde x}\le  \|\tilde  v\|_U \le (1+\epsilon)\|\Phi_{\tilde x}^{-1}(\tilde v)\|_{\tilde x} \quad \text{for all } \tilde x\in U \text{ and all } \tilde v\in\mathbb R^d,
    \\
    (1-\epsilon)\|\hat \Phi_{\tilde x}^{-1}(\tilde v)\|_{\tilde x}\le  \|\tilde  v\|_V \le (1+\epsilon)\|\hat \Phi_{\tilde x}^{-1}(\tilde v)\|_{\tilde x} \quad \text{for all } \tilde x\in V \text{ and all } \tilde v\in\mathbb R^d,
\end{align*}
by shrinking $U$ and $V$.   
\end{rmk}

\begin{rmk}
    Note that the H\"older regularity of $h$ and $\tilde h$ in Theorem \ref{thm2} does not hold for all   $0<\tilde \alpha\le \alpha \beta$ and all $\alpha$ satisfying \eqref{expo.alp.}. They are chosen to such that 
    all conditions $0<\tilde \alpha\le \alpha \beta$, \eqref{expo.alp.} and  \eqref{less-one} hold.
\end{rmk}

Using cut-off technique and partition of unity, Theorems \ref{thm1} and \ref{thm2} yield the local version of Hartman-Grobman Theorem on smooth vector bundle.

\begin{thm}\label{thm3}
    Assume that $(E, X, \pi)$ is a $C^2$ smooth vector bundle of rank $d$, where $X$ is a $C^1$ compact Riemannian manifold, $f:X\to X$ is a $C^1$ diffeomorphism, and  $F: E \to E$ is a $C^2$ diffeomorphism covering $f$ and its derivative $A$ at zero section is hyperbolic. Then $F$ is locally topologically conjugate to its linear part $A$ near  zero section. In addition, the local conjugacy both are H\"older continuous along base and fiber direction.
\end{thm}

\begin{rmk}
The classical Hartman-Grobman theorem says that a $C^1$ local diffeomorphism admits $C^0$ local linearization. While in our vector bundle setting we require that $F:E\to E$ is $C^2$ because we need that the partial derivative of $F$ along fiber direction is also differentiable along base direction in order  to satisfy conditions of Theorem \ref{thm2}.
\end{rmk}


\section{Uniformly equivalent continuous norm}

In this section, we construct a new continuous norm $\|\cdot\|^*$ on $E$ under which we have \eqref{contrac.} and \eqref{expans.}.
\begin{prop}
    There exists a continuous norm $\|\cdot\|^*$ on $E$ under which we have \eqref{contrac.} and \eqref{expans.}. In addition, the new continuous norm is equivalent to the original one. Namely, there exist constants $r_1, r_2>0$ independent of $x$ such that 
\[ r_1\| v\|_x \le \| v\|_x^* \le r_2\| v\|_x \qquad \text{for each } v\in E_x \text{ and each } x\in X.\]
\end{prop}

\begin{proof}
 For each $x\in X$ and $v\in E_x$, there exist unique $v^s \in E^s_x$ and $v^u\in E_x^u$ such that $v=v^s+v^u$. Define a projection $P: E\to E$ to be 
\[ P(x,v):= (x,v^s)\qquad \text{for each } x\in X \text{ and each } v\in E_x.\]
We claim that $P$ is continuous on $E$. Indeed, letting $\{(x_n,v_n)\}_{n\ge 0}\subset E$ converge to $(x_\infty,v_\infty)$, we need to show that $\{(x_n,v_n^s)=P(x_n,v_n)\}_{n\ge 0}\subset E$ converge to $(x_\infty, P_xv_\infty:=v^s_\infty)$. Since $E=E^s\bigoplus E^u$ is a continuous splitting, we have that
there is a neighborhood $U\subset X$ of $x_\infty$ such that there exist a local continuous frame $(\sigma_i)_{i=1}^k$ over $U$ for $E^s$
and a local continuous frame $(\sigma_i)_{i=k+1}^d$ over $U$ for $E^u$. Moreover, $(\sigma_i)_{i=1}^d$ forms a continuous local frame for $E$. We may assume that $\{(x_n,v_n)\}_{n\ge 0} \subset \pi^{-1}(U)$ since it converges to $(x_\infty,v_\infty)$. Then for each $n\ge 0$
\[v_n= \sum_{i=1}^da^i_n\sigma_i|_{x_n},\]
where $\{a_n^i\}_{i=1}^d$ are some real numbers, and 
\[v_\infty= \sum_{i=1}^da^i_\infty\sigma_i|_{x_\infty},\]
for some real numbers $\{a_\infty^i\}_{i=1}^d$. It is easy to see that the convergence $(x_n,v_n)\to (x_\infty,v_\infty)$ as $n\to \infty$ is equivalent to $x_n\to x_\infty$ and $a_n^i \to a_\infty^i$ as $n\to \infty$ for all $i=1,\cdots,d$, and that for all $n\ge 0$
\[ v^s_n := \sum_{i=1}^ka^i_n\sigma_i|_{x_n} \qquad \text{and}\qquad v_\infty^s= \sum_{i=1}^ka^i_\infty\sigma_i|_{x_\infty}. \]
Consequently, $\{(x_n,v_n^s)=P(x_n,v_n)\}_{n\ge 0}\subset E$ converges to $(x_\infty, P_xv_\infty:=v^s_\infty)$. The claim is proved. 

Define a projection $Q: E\to E$ to be 
\[Q (x,v) := (x,v-P_x v)=(x,v^u) \qquad \text{ for each } (x,v) \in E. \]
Obviously, $Q$ is continuous on $E$ because so is $P$. Now define a norm $\|\cdot\|_x'$ on $E_x$ for each $x\in X$ by 
\[\|v\|_x':= \max \{ \|v^s\|_x, \|v^u\|_x\}. \]
Apparently, $\|\cdot\|'$ is continuous on $E$ because $P$ and $Q$ both are continuous. We claim that there exists a constant $\tilde C>0$ independent of $x$ such that 
\[ \frac{1}{2}\| v\|_x  \le \| v\|_x' \le \tilde C\| v\|_x \qquad \text{ for each } x\in X \text{ and each } v\in E_x.\]
In fact, given fixed $x\in X$, for each $v\in E_x$, $\|v\|_x = \| v^s+v^u\|_x \le \|v^s\|_x+\|v^u\|_x \le 2 \|v\|'_x$, which shows that $\frac{1}{2}\|v\|_x \le \|v\|'_x$. For the converse direction, We need the following lemma to help, whose proof is given after finishing this proof.
\begin{lm}\label{lm-unif.-equia.-norm}
There exist two constants $C_1, C_2>0$ independent of $x$ such that
\begin{equation}\label{loc.equi.Eucli.}
    C_1 \| v\| \le \| \Phi^{-1}_x (v)\|_x \le C_2 \|v\| \qquad \text{ for each } (x,v)\in  V\times \mathbb R^d,
\end{equation}    
where $\|\cdot\|$ stands for the Euclidean norm of $\mathbb R^d$ and $\Phi: \pi^{-1}(V) \to V\times \mathbb R^d$ is a continuous local trivialization near $x$.
\end{lm}  
Since $X$ is compact and $P$ is continuous,  there exists a constant $C_3>0$ independent of $x$ such that for each $x\in X$ there is a local trivialization $\Phi: \pi^{-1}(V) \to V\times \mathbb R^d$ satisfying
\begin{equation}
    \| \tilde P_x (v)\| \le C_3 \| v\| \qquad \text{for each } (x,v)\in V\times \mathbb R^d,
\end{equation}
where $\|\cdot\|$ stands for the Euclidean norm of $\mathbb R^d$ and $\tilde P_x(v)$ is defined to be $\tilde P_x(v):= \Phi_x \circ P_x \circ \Phi_x^{-1}(v)$. Therefore, for each $x\in X$ and each $v\in E_x$, there exist a local trivialization $\Phi$ near $x$ and $\tilde v\in\mathbb R^n$ such that  
\begin{align*}
    \|P_x (v)\|_x&=  \| P_x \circ \Phi_x^{-1}(\tilde v)\|'_x 
    \\
    &= \| \Phi_x^{-1}\circ \Phi_x\circ P_x \circ \Phi_x^{-1}(\tilde v)\|_x
    \\
    & = \| \Phi_x^{-1}(\tilde P_x(\tilde v))\|_x
    \\
    & \le C_2 \|\tilde P_x(\tilde v)\| \le C_2C_3 \| \tilde v\| \le \frac{C_2C_3}{C_1}\|\Phi_x^{-1}(\tilde v)\|_x
    \\
    & =: \tilde C_1 \|v\|_x.
\end{align*}
Analogously, we are able to show that there exists a constant $\tilde C_2$ independent of $x$ such that 
\[ \|Q_x(v)\|_x \le \tilde C_2 \| v\|_x\qquad \text{for each } x\in X \text{ and each } v\in E_x.\]
Thus, we have that  for each $x\in X$ and each $v\in E_x$
\[ \|v\|_x' = \max \{ \|v^s\|_x, \|v^u\|_x\} = \max \{ \|P_x(v)\|_x, \|Q_x(v)\|_x \} \le \max \{ \tilde C_1 , \tilde C_2\} \| v\|_x ,\]
where proves the claim with $\tilde C:= \max \{ \tilde C_1 , \tilde C_2\}.$ By the definition of $\|\cdot\|_x'$, we see that for all \( n \geq 0 \):
  \[
  \|A^n(v)\|_{f^n(x)}' \leq C \lambda^n \|v\|_x' \quad \text{for } v \in E^s_x,
  \]
  \[
  \|A^{-n}(v)\|_{f^{-1}(x)}' \leq C \lambda^n \|v\|_x' \quad \text{for } v \in E^u_x.
  \]

Let $N\in \mathbb{N}$ be sufficiently large such that 
\[C\lambda^N<1.\]
Define 
\begin{align}
    \|v\|_x^* &:= \sum_{n=0}^{N-1} \| A^n(v)\|_{f^n(x)}'  \qquad \text{for } v\in E_x^s,
    \notag\\
    \| v\|_x^* &:= \sum_{n=0}^{N-1} \|A^{-n}(v)\|_{f^{-n}(x)}'\qquad \text{for } v\in E_x^u,
    \notag\\
    \| v\|_x^* =\|v^s+v^u\|_x^* &:= \max \{ \|v^s\|_x^*, \|v^u\|_x^*\} \qquad \text{for } v\in E_x,
    \label{norm-star}
\end{align}
where $v^s\in E^s_x$ and $v^u\in E_x^u$. It is clear to check that for each $x\in X$ and $v\in E_x$
\begin{align*}
    \|v\|_x' &= \max \{ \|v^s\|_x', \|v^u\|_x'\} \le \max \{ \|v^s\|_x^*, \|v^u\|_x^*\}=\|v\|_x^*,
    \\
    \|v\|_x^* & \le \|v^s\|_x^* + \|v^u\|_x^* \le \Big\{ 1+\sum_{n=1}^{N-1}C\lambda^n\Big\}(\|v^s\|'_x+\|v^u\|'_x)\le B\|v\|_x',
\end{align*}
where 
\[1< B:=2\Big\{ 1+\sum_{n=1}^{N-1}C\lambda^n\Big\}<\infty.\]
This immediately yields that the norm $\|\cdot\|^*$ is uniformly equivalent to the original one $\|\cdot\|$. It is obvious that $\|\cdot\|^*$ is continuous on $E$.
Thus,
for $v\in E_x^s$,  we have 
\begin{align*}
    \|A(v)\|_{f(x)}^*  &= \sum_{n=1}^{N} \| A^n(v)\|_{f^n(x)}'  
   \\
& = \| v\|_x^*-  \| v\|_x' + \|A^N(v)\|_{f^N(x)}'
\\
& \le \| v\|_x^* -  \| v\|_x' + C\lambda^N \| v\|_x'
\\
& \le \Big\{ 1- \frac{1-C\lambda^N}{B} \Big\} \| v\|_x^*
\\
& =: \tau \|v\|_x^*.
\end{align*}
Clearly, $0<\tau <1$.
For $v\in E_x^u$,  we have
\begin{align*}
    \| A^{-1}(v)\|_{f^{-1}(x)}^* &= \sum_{n=1}^{N} \| A^{-n}(v)\|_{f^{-n}(x)}'  
   \\
& = \| v\|_x^*-  \| v\|_x' + \|A^{-N}(v)\|_{f^{-N}(x)}'
\\
& \le \| v\|_x^* -  \| v\|_x' + C\lambda^N \| v\|_x'
\\
& \le \Big\{ 1- \frac{1-C\lambda^N}{B} \Big\} \| v\|_x^*
\\
& =: \tau \|v\|_x^*.
\end{align*}
This proves \eqref{contrac.} and \eqref{expans.}.  The proof is complete.
\end{proof}

\begin{proof}[Proof of Lemma \ref{lm-unif.-equia.-norm}]
    Consider a continuous local trivialization for $E$
\[ \Phi : \pi^{-1}(U) \to U\times \mathbb R^d,\]
where $X \supset U\ni x$ is a neighborhood of $x$. Then we get a continuous norm on $U\times \mathbb R^d$ defined by
\[ \| \cdot\| \circ \Phi^{-1} : U\times \mathbb R^d \to \mathbb R_{\ge 0},\]
where $\mathbb R_{\ge 0}$ denotes the set of nonnegative real numbers.
Let $V$ be a neighborhood of $x$ such that $x\in V\subset \bar V \subset U$. Then $\|\cdot\|\circ \Phi^{-1}$ admits minimum and maximum on $\bar V\times \mathbb S^{d-1}$ because $\bar V\times \mathbb S^{d-1}$ is a compact subset of $U\times \mathbb R^d$. As a result, there exist constant $C_1, C_2>0$ such that for each $(x,v)\in \bar V\times \mathbb S^{d-1}$
\[0<C_1\le \|\Phi^{-1}_x(v)\|_x \le C_2 < \infty,\]
which gives rise to \eqref{loc.equi.Eucli.},
where $\|\cdot\|$ denotes the Euclidean norm of $\mathbb R^d$. Since $X$ is compact, there are finite many open sets $V_1,\cdots,V_\ell$ such that $\{ \pi^{-1}(V_i)\}_{i=1}^\ell$ covers $E$ and local trivializations $\{\Phi_i\}_{i=1}^\ell$ over $\{ \pi^{-1}(V_i)\}_{i=1}^\ell$ respectively such that \eqref{loc.equi.Eucli.} holds. Consequently, for each $x\in X$, there exists a local trivialization $\Phi: \pi^{-1} (V) \to V\times \mathbb R^d$ satisfying $\eqref{loc.equi.Eucli.}$ with $C_1$ and $C_2$ independent of $x$, where $V$ is a neighborhood of $x$.
\end{proof}

For simplicity of symbols, we still use $\|\cdot\|$ to stand for $\|\cdot\|^*$ on $E$ and omit the subindex $x$ of $\|\cdot\|_x$ when there is no confusion. In other words, we assume that there exists a continuous norm on $E$ and a number $0<\tau<1$ such that  for each $x\in X$
\[ \| A_x(v)\|\le  \tau \|v\| \quad \text{for } v\in E_x^s,\]
\[ \| A_x^{-1}(v)\|\le  \tau\|v\| \quad \text{for } v\in E_x^u.\]

\section{Proofs of Theorems \ref{thm1} and \ref{thm2}}

In order to prove Theorem \ref{thm1}, we need the following main lemma.
\begin{lm}\label{lm-analog-push}
Let $A$ be assumed as above and $\phi_x, \psi_x \in \mathcal{C}^{0,1}_b(E;\delta)$. Then
        equation   
   \begin{equation}\label{pertu.linea}
 {\mathcal H}_{f(x)} \circ (A_x+\phi_x) =( A_x +\psi_x) \circ {\mathcal H}_x \qquad \text{for all} \quad x\in X,  
  \end{equation}
has a homeomorphism solution ${\mathcal H}: E\to E$.
 Moreover, ${\mathcal H}_x: E_x\to E_x$ is chosen to be $\Id + h_x$ for all $x\in X$, where $h_x : E_x \to E_x$ is  bounded and H\"older continuous with exponent 
 \[\alpha\in \Big(0, \min \Big\{ \frac{\log(\tau+\delta)}{\log(\ell-\delta)}, \frac{\log \{(1-\delta)/\tau\}}{\log(L+\delta)} \Big\}\Big). \]
  The inverse of $\mathcal H_x$ is also of the form $\Id + \tilde h_x$ for all $x\in X$, where $\tilde h_x : E_x \to E_x$ is also  bounded and H\"older continuous with the same exponent $\alpha$.
\end{lm}

\begin{proof}
  Let $P_x$ be the projection from $E_x$ to $E_x^s$ for each $x\in X$.
Letting ${\mathcal H}_x:=\Id+h_x$, the equation \eqref{pertu.linea} is written as 
\[ (\Id + h_{f(x)}) \circ (A_x+\phi_x) = (A_x+\psi_x) \circ (\Id + h_x),\]
which is reduced to 
\[ \phi_x + h_{f(x)} \circ (A_x+\phi_x)= A_x\circ h_x + \psi_x\circ (\Id+h_x).\]
Projecting this equation into the stable invariant bundle and unstable invariant bundle respectively, we obtain that
\begin{align*}
   & P_x \circ \phi_x + P_x \circ  h_{f(x)} \circ (A_x+\phi_x) = A_x^s \circ h_x+P_x\circ \psi_x\circ (\Id+h_x),
   \\
   & Q_x \circ \phi_x + Q_x \circ  h_{f(x)} \circ (A_x+\phi_x) = A_x^u \circ h_x + Q_x \circ \psi_x\circ (\Id+h_x),
\end{align*}
where $Q_x:= \Id - P_x$, $A_x^s := P_x \circ A_x$, and $A_x^u := Q_x \circ A_x$ for all $x\in X$. Thus, let $\phi^s_x:=P_x\circ \phi_x$, $\phi^u_x:=Q_x\circ \phi_x$, $\psi^s_x:=P_x\circ \psi_x$ and $\psi^u_x:=Q_x\circ \psi_x
$ for each $x\in X$ and we see that the equation \eqref{pertu.linea} reduces to 
\begin{align}
    h_x &= P_x \circ h_x + Q_x \circ h_x 
    \notag
    \\
    &= \{A_{f^{-1}(x)}^s \circ h_{f^{-1}(x)} + \psi^s_{f^{-1}(x)}\circ (\Id+h_{f^{-1}(x)})  - \phi^s_{f^{-1}(x)} \}  \circ (A_{f^{-1}(x)}+\phi_{f^{-1}(x)})^{-1}   \notag
    \\
    & \qquad + (A^u_x)^{-1} \circ \{ \phi^u_x + Q_x \circ  h_{f(x)} \circ (A_x+\phi_x) - \psi^u_x \circ (\Id+h_x)\}.\label{reduce.equa}
\end{align}

Consider a space 
\begin{align*}
\mathcal C_b(E)&:= \Big\{ \sigma: E\to E | \sigma \text{ is continuous,  } \sigma_x|_{E_x}: E_x \to E_x \text{ is bounded  }
\\
& \qquad \text{under the norm  defined by \eqref{norm-star} and } \sup_{x\in X}\sup_{v\in E_x}\|\sigma_x(v)\| <\infty\Big\}.
\end{align*}
It is easy to check that $\mathcal C_b(E)$ is a Banach space equipped with the norm 
\[ \|\sigma\|:= \sup_{x\in X}\sup_{v\in E_x}\|\sigma_x(v)\|\]
for each $\sigma\in \mathcal C_b(E)$.
Define an operator $\mathcal T: \mathcal C_b(E) \to \mathcal C_b(E)$ by defining $(\mathcal{T}\sigma)_x $ to be the right hand side of \eqref{reduce.equa} for all $\sigma\in \mathcal C_b(E)$ and for all $x\in X$ with $h$ replaced with $\sigma$, that is,
\begin{align}
(\mathcal{T}\sigma)_x &:=  \{A_{f^{-1}(x)}^s \circ \sigma_{f^{-1}(x)} +\psi^s_{f^{-1}(x)}\circ (\Id+\sigma_{f^{-1}(x)})  - \phi^s_{f^{-1}(x)} \}  \circ (A_{f^{-1}(x)}+\phi_{f^{-1}(x)})^{-1}   \notag
    \\
    & \qquad + (A^u_x)^{-1} \circ \{\phi^u_x + Q_x \circ  \sigma_{f(x)} \circ (A_x+\phi_x) - \psi^u_x \circ (\Id+\sigma_x) \}.\label{operator}
\end{align}
Apparently, a fixed point of $\mathcal T$ is a solution of equation \eqref{reduce.equa}. 

Now we are going to find a fixed point of $\mathcal T$. First of all, we verify that $\mathcal T: \mathcal C_b(E) \to \mathcal C_b(E)$ is well defined. To see that $\mathcal T(\sigma)$ is continuous on $E$ for each $\sigma\in E$, it suffices to show that $\varphi: E\to E$ defined by 
\[\varphi(x,v) = (f^{-1}(x), (A_{f^{-1}(x)}+\phi_{f^{-1}(x)})^{-1}(v) ) \qquad \text{ for } x\in X \text{ and } v\in E_x\]
is continuous. $f^{-1}$ is continuous by our assumption. Let $(x_n,v_n) \in U\times \mathbb R^n$ for all $n\ge 1$ and $(x_n,v_n)\to (x_\infty,v_\infty)\in U\times \mathbb R^n$ as $n\to \infty$, where $U \subset X$ is a neighborhood of $x_\infty$. We need to prove that 
\[\Phi \circ (A_{f^{-1}(x_n)}+\phi_{f^{-1}(x_n)})^{-1} \circ \Phi^{-1}(x_n,v_n) \to   \Phi \circ (A_{f^{-1}(x_\infty)}+\phi_{f^{-1}(x_\infty)})^{-1} \circ \Phi^{-1}(x_\infty,v_\infty) \]
as $n\to \infty$. Let 
\begin{align*}
    \tilde v_n &:= \Phi_{f^{-1}(x_n)} \circ (A_{f^{-1}(x_n)}+\phi_{f^{-1}(x_n)})^{-1} \circ \Phi_{x_n}^{-1}(v_n) \qquad \text{and}
    \\
    \tilde v_\infty &:= \Phi_{f^{-1}(x_\infty)} \circ (A_{f^{-1}(x_\infty)}+\phi_{f^{-1}(x_\infty)})^{-1} \circ \Phi_{x_\infty}^{-1}(v_\infty).
\end{align*}
If $\|\tilde v_n\|$ (or $\|\tilde v_{n_k}\|$)  approaches to $\infty$ as $n\to \infty$ ($k\to \infty$), where $\{n_k\}_{k\ge 1}$ is a subsequence of $\mathbb N$, then by \eqref{loc.equi.Eucli.} and the assumption that Lip$(\phi_x) \le  \delta< \ell$ for all $x\in X$
\[ \| v_n\| = \|\Phi_{f^{-1}(x_n)} \circ (A_{f^{-1}(x_n)} + \phi_{f^{-1}(x_n)})\circ \Phi_{x_n}^{-1}(\tilde v_n)\|\ge \kappa \|\tilde v_n\|,\]
where $\kappa>0$ is a constant, which implies that $\|v_n\| \to \infty$ as $n\to\infty$. While $v_n\to v_\infty$ as $n\to \infty$,  it is a contradiction. 
If $\tilde v_n \to \tilde v_\infty^* \ne \tilde v_\infty$ as $n\to \infty$, then by the continuity of $A+\phi$ we obtain that 
\[   \Phi_{f^{-1}(x_n)} \circ (A_{f^{-1}(x_n)}+\phi_{f^{-1}(x_n)}) \circ \Phi_{x_n}^{-1}(\tilde v_n) 
\to \Phi_{f^{-1}(x_\infty)} \circ (A_{f^{-1}(x_\infty)}+\phi_{f^{-1}(x_\infty)}) \circ \Phi_{x_\infty}^{-1}(\tilde v^*_\infty)\]
as $n\to \infty$. But 
\[\Phi_{f^{-1}(x_n)} \circ (A_{f^{-1}(x_n)}+\phi_{f^{-1}(x_n)}) \circ \Phi_{x_n}^{-1}(\tilde v_n) = v_n
\to v_\infty \quad \text{ as } \quad n\to \infty,\]
which means that $\tilde v_\infty$ and $\tilde v^*_\infty$ both are pre-image of $v_\infty$ under a local trivialization of $A+\phi$. This contradicts to the assumption that $A+\phi$ is a bijection.
Since $\sigma_x$, $\phi_x$, $\psi_x$ and $P_x$ all are bounded for each $x\in X$, we see that 
$(\mathcal{T}\sigma)_x$ is also bounded for each $x\in X$. 
Moreover, 
\begin{align*}
    \sup_{v\in E_x} \| (\mathcal{T}\sigma)_x(v)\|_x \le 2\lambda \|\sigma\| + (1+\lambda)B_\phi + (1+\lambda) B_\psi \qquad \text{for all $x\in X$,}
\end{align*}
where $B_\phi$ and $B_\psi$ are defined by \eqref{bound-perturb.}. It is proved that $\mathcal T: \mathcal C_b(E)\to \mathcal C_b(E)$ is well defined.

Next, we show that $\mathcal T$ is a contraction. For $\sigma^1,\sigma^2
\in \mathcal C_b(E)$, by \eqref{operator}, the definition of norm, \eqref{contrac.} and \eqref{expans.}, and the assumption that $\phi,\psi\in \mathcal{C}^{0,1}_b(E;\delta)$, we see that
\begin{align*}
   &\| P_x(\mathcal T\sigma^1)_x(v) -  P_x(\mathcal T\sigma^1)_x(v)\|
   \\
   &\le \tau \| \sigma^1_{f^{-1}(x)} \circ (A_{f^{-1}(x)}+\phi_{f^{-1}(x)})^{-1}(v) - \sigma^2_{f^{-1}(x)} \circ (A_{f^{-1}(x)}+\phi_{f^{-1}(x)})^{-1}(v)\|
   \\
   & \qquad + \delta  \| \sigma^1_{f^{-1}(x)} \circ (A_{f^{-1}(x)}+\phi_{f^{-1}(x)})^{-1}(v) - \sigma^2_{f^{-1}(x)} \circ (A_{f^{-1}(x)}+\phi_{f^{-1}(x)})^{-1}(v)\|
\end{align*}
and that 
\begin{align*}
   &\| Q_x(\mathcal T\sigma^1)_x(v) -  Q_x(\mathcal T\sigma^1)_x(v)\|  
   \\
   & \le \tau \| \sigma^1_{f(x)}\circ (A_x+\phi_x)(v) - \sigma^2_{f(x)}\circ (A_x+\phi_x)(v)\|
  + \tau \delta \| \sigma^1_x(v)-\sigma^1_x(v)\|.
\end{align*}
It follows that 
\begin{align*}
    &\| \mathcal T \sigma^1 - \mathcal T\sigma^2\|= \sup_{x\in X} \sup_{v\in E_x} \| (\mathcal T \sigma^1)_x(v) - (\mathcal T\sigma^2)_x(v)\|
    \\
    & = \sup_{x\in X} \sup_{v\in E_x}  \max \{ \| P_x(\mathcal T\sigma^1)_x(v) -  P_x(\mathcal T\sigma^1)_x(v)\|, \| Q_x(\mathcal T\sigma^1)_x(v) -  Q_x(\mathcal T\sigma^1)_x(v)\|  \}
    \\
    & \le (\tau +\delta)\| \sigma^1 -\sigma^2\|.
\end{align*}
This shows that $\mathcal T$ is a contraction since $\tau+\delta <1$ by our assumption. Consequently, by the Banach  contraction principle, $\mathcal T$ has a unique fixed point $h$ in $\mathcal C_b(E)$, which is a solution of equation \eqref{reduce.equa}. Thus, equation \eqref{pertu.linea} has a continuous solution $\mathcal H$ such that $\mathcal H_x: E_x\to E_x$ has the form $\Id + h_x$ with $h_x$ being bounded.

Now, let us consider a closed subset of $\mathcal C_b(E)$,
\begin{align*}
\mathcal C_b^\alpha(E)&:= \{ \sigma\in \mathcal C_b(E) : \|\sigma_x(v_1) - \sigma_x(v_2)\| \le M \| v_1-v_2\|^\alpha
\\
&\qquad\text{ for each $x\in X$ and all $v_1,v_2\in E_x$} \}, 
\end{align*}
where $M>0$ and $0<\alpha<1$ are two constants independent of $x$. We will show that $\mathcal T$ maps $\mathcal C_b^\alpha(E)$ into itself.
Indeed, if $\|v_1 - v_2\| \ge 1$, then
\begin{align*}
    &\| P_x (\mathcal T\sigma)_x(v_1) - P_x (\mathcal T\sigma)_x(v_2)\| 
    \\
    &\le \tau M\| (A_{f^{-1}(x)}+\phi_{f^{-1}(x)})^{-1}(v_1) -(A_{f^{-1}(x)}+\phi_{f^{-1}(x)})^{-1}(v_2)\|^\alpha 
    \\
    & \qquad+( 2B_\psi +2 B_\phi) \| v_1-v_2\|^\alpha
    \\
    & \le \Big\{ \frac{\tau M }{(\ell-\delta)^\alpha} + 2B_\psi+2B_\phi\Big\} \| v_1-v_2\|^\alpha
\end{align*}
and 
\begin{align*}
     \| Q_x (\mathcal T\sigma)_x(v_1) - Q_x (\mathcal T\sigma)_x(v_2)\|    
      \le (2\tau B_\phi + \tau M (L+\delta)^\alpha + 2\tau B_\psi)  \| v_1-v_2\|^\alpha;
\end{align*}
if $\| v_1-v_2\| <1$, then
\begin{align*}
       \| P_x (\mathcal T\sigma)_x(v_1) - P_x (\mathcal T\sigma)_x(v_2)\| 
       \le  \Big\{\frac{\tau M }{(\ell-\delta)^\alpha} +\frac{2\delta}{\ell-\delta} + \frac{\delta M}{(\ell-\delta)^\alpha}\Big\} \| v_1-v_2\|^\alpha
\end{align*}
and 
\begin{equation*}
     \| Q_x (\mathcal T\sigma)_x(v_1) - Q_x (\mathcal T\sigma)_x(v_2)\|    \le (\tau \delta + \tau M (L+\delta)^\alpha + \delta + \delta M ) \| v_1-v_2\|^\alpha.
\end{equation*}
It follows that 
\[\|(\mathcal T\sigma)_x(v_1) - (\mathcal T\sigma)_x(v_2)\|  \le K \| v_1-v_2\|^\alpha, \]
where 
\begin{align*}
    K&:=\max \Big\{ \frac{\tau M }{(\ell-\delta)^\alpha} + 2B_\psi+2B_\phi, \quad 2\tau B_\phi + \tau M (L+\delta)^\alpha + 2\tau B_\psi,
    \\
    & \qquad\frac{\tau M }{(\ell-\delta)^\alpha} +\frac{2\delta}{\ell-\delta} + \frac{\delta M}{(\ell-\delta)^\alpha}, \quad \tau \delta + \tau M (L+\delta)^\alpha + \delta + \delta M
    \Big\}.
\end{align*}
To ensure $K \le M$ and existence of $M$, we need $\alpha $ to satisfy 
\[ \alpha < \min \Big\{ \frac{\log(\tau+\delta)}{\log(\ell-\delta)}, \frac{\log \{(1-\delta)/\tau\}}{\log(L+\delta)} \Big\}. \]
Choose 
\begin{align*}
    M&\ge \max \Big\{ \frac{2(B_\psi+B_\phi)}{1-\tau/(\ell-\delta)^\alpha}, \quad \frac{2\tau (B_\phi+B_\psi)}{1- \tau (L+\delta)^\alpha},
    \\
    & \qquad  \frac{2\delta}{(\ell-\delta )(1- (\tau+\delta)/(\ell-\delta)^\alpha)}, \quad \frac{(\tau+1)\delta}{1-\tau(L+\delta)^\alpha - \delta} \Big\}.
\end{align*}
Then $K\le M$. This shows that $\mathcal T$ maps $\mathcal C^\alpha_b(E)$ into itself. As a consequence, the unique solution in $\mathcal C_b(E)$ is actually contained in $\mathcal C^\alpha_b(E)$.

Now, we consider the inverse of $\mathcal H$. Switch places of $\phi_x$ and $\psi_x$ in equation \eqref{pertu.linea}. Using the above procedure, we can have that equation 
   \begin{equation}\label{swith}
 {\mathcal H}_{f(x)} \circ (A_x+\psi_x) =( A_x +\phi_x) \circ {\mathcal H}_x \qquad \text{for all} \quad x\in X,  
  \end{equation}
also has a continuous solution $\tilde{\mathcal{H}}$ such that $\tilde{\mathcal{H}}_x : E_x \to E_x$ is also of the form $\Id + \tilde h_x$, where $\tilde h \in \mathcal C^\alpha_b(E)$ and is unique in $\mathcal C_b(E)$. We claim that $\tilde{\mathcal H}$ is exactly the inverse of $\mathcal H$. In fact, let us consider equation 
   \begin{equation}\label{triv.-equa.}
 {\mathcal H}_{f(x)} \circ (A_x+\phi_x) =( A_x +\phi_x) \circ {\mathcal H}_x \qquad \text{for all} \quad x\in X.
  \end{equation}
Obviously, it has a trivial solution $\Id: E\to E$ but has another solution $\tilde{\mathcal H} \circ \mathcal H$ because equations \eqref{pertu.linea} and \eqref{swith} yield 
\begin{align*}
\tilde{\mathcal H}_{f(x)} \circ \mathcal H_{f(x)} \circ (A_x+\phi_x)=   \tilde{\mathcal H}_{f(x)} \circ (A_x+\psi_x) \circ \mathcal H_x = (A_x+\phi_x) \circ \tilde{\mathcal H}_{x} \circ \mathcal H_{x}.
\end{align*}
It is easy to see that
\[\tilde{\mathcal H}_{x} \circ \mathcal H_{x}= (\Id + \tilde h_{x}) \circ (\Id + h_{x}) = \Id + h_x + \tilde h_x\circ (\Id+h_x)\]
and that
\[\sup_{x\in X} \sup_{v\in E_x} \|h_x(v)+\tilde h_x(v+h_x(v)\| <\infty. \] 
This says that $h_x + \tilde h_x\circ (\Id+h_x) \in \mathcal C_b(E)$. By the uniqueness of solution of equation \eqref{triv.-equa.} in the space $\mathcal C_b(E)$, we conclude that $h_x + \tilde h_x\circ (\Id+h_x) $ must be zero. In other words, 
\begin{equation}\label{invers.1}
    \tilde{\mathcal H}_{x} \circ \mathcal H_{x} = \Id.
\end{equation}
Take into account equation 
   \begin{equation*}
 {\mathcal H}_{f(x)} \circ (A_x+\psi_x) =( A_x +\psi_x) \circ {\mathcal H}_x \qquad \text{for all} \quad x\in X.
  \end{equation*}
Using similar arguments, we derive that 
\begin{equation}\label{invers.2}
\mathcal H_{x} \circ       \tilde{\mathcal H}_{x} = \Id.
\end{equation}
Thus, equalities \eqref{invers.1} and \eqref{invers.2} imply the claim. This proves Lemma \ref{lm-analog-push}.
\end{proof}

As a special case, Theorem \ref{thm1} is immediately obtained by Lemma \ref{lm-analog-push}.

The following lemma is helpful to the proof of Theorem \ref{thm2}.
\begin{lm}\label{lm-Holder-along.base}
    Under the assumptions of Theorem \ref{thm2}, we also assume that $\psi_x \in \mathcal{C}^{0,1}_b(E;\delta)$,
    \begin{equation}
    \|\hat\Phi_{f(x)} \circ \psi_x \circ \Phi_x^{-1}(\tilde v) - \hat \Phi_{f(\tilde x)} \circ \psi_{\tilde x} \circ \Phi_{\tilde x}^{-1}(\tilde v) \|\le M_\phi d(x,\tilde x)^\beta\| \tilde v\|,   \end{equation}
and $B_\psi$ is sufficiently small.
    Then  the functions $h$ and $\tilde h$ in Lemma \ref{lm-analog-push} both are H\"older continuous along base space $X$ with some exponent $0<\tilde \alpha \le \alpha\beta$.
\end{lm}

\begin{proof}
The idea is to show that $\mathcal{T}$ is a self-mapping on a closed subset of $\mathcal C_b^\alpha(E)$, which consists of maps $\sigma: E\to E$ both H\"older continuous along fiber direction and base direction.

Consider a closed subset of $\mathcal C_b^\alpha(E)$,
\begin{align*}
 \mathcal C_b^{\alpha,\tilde \alpha}(E) &:= \{\sigma \in \mathcal C_b^\alpha(E): \text{ for each $x\in X$ there exists a neighborhood $U$ of $x$ } 
 \\
 & \qquad \text{ such that for each $\tilde v\in\mathbb R^d$ and each $\tilde x \in U$,}
 \\
 &\qquad
\| \Phi\circ\sigma\circ \Phi^{-1}(x,\tilde v) - \Phi\circ\sigma\circ \Phi^{-1}(\tilde x,\tilde v) \| \le \tilde M d(x,\tilde x)^{\tilde \alpha} \| \tilde v\|^\alpha \},
\end{align*}
where $\Phi: \pi^{-1} (U) \to U\times \mathbb R^d$ is a local trivialization for $E$, $\| \cdot\|$ stands for the Euclidean norm on $\mathbb R^d$, $\tilde M>0$ and $0<\tilde \alpha <1$, independent of $x,\tilde x$ and $\tilde v$, are two constants to be determined, and $d$ denotes the distance on $X$. Next, we show that $\mathcal T$ maps $\mathcal C_b^{\alpha,\tilde \alpha}(E) $ into itself for some $\tilde M>0$ and $0<\tilde \alpha<1$. Indeed, letting 
\[ B_{f^{-1}(x)}:= \{A_{f^{-1}(x)}^s \circ \sigma_{f^{-1}(x)} +\psi^s_{f^{-1}(x)}\circ (\Id+\sigma_{f^{-1}(x)})  - \phi^s_{f^{-1}(x)} \}\circ (A_{f^{-1}(x)}+\phi_{f^{-1}(x)})^{-1}  \]
for $x\in X$, from \eqref{operator}
we see that 
\begin{align*}
  &\| \Phi \circ P_x\circ(\mathcal T \sigma) \circ \Phi^{-1} (x,\tilde v) - \Phi \circ P_{\tilde x}\circ(\mathcal T \sigma) \circ \Phi^{-1} (\tilde x,\tilde v) \| 
  \\
  &= \| \Phi_x \circ B_{f^{-1}(x)}  \circ \Phi_x^{-1}(\tilde v) -  \Phi_{\tilde x} \circ B_{f^{-1}(\tilde x)}  \circ \Phi_{\tilde x}^{-1}(\tilde v) \|
  \\
  &\le  D_1 + D_2 +D_3,
\end{align*}
where 
\begin{align*}
    D_1&:= \| \Phi_x\circ A_{f^{-1}(x)}^s \circ \sigma_{f^{-1}(x)} \circ  (A_{f^{-1}(x)}+\phi_{f^{-1}(x)})^{-1} \circ \Phi_x^{-1}(\tilde v) 
\\
& \qquad-   \Phi_{\tilde x}\circ A_{f^{-1}(\tilde x)}^s \circ \sigma_{f^{-1}(\tilde x)} \circ(A_{f^{-1}(\tilde x)}+\phi_{f^{-1}(\tilde x)})^{-1} \circ \Phi_{\tilde x}^{-1}(\tilde v) \|,
\\
D_2 &:= \| \Phi_x \circ \psi^s_{f^{-1}(x)} \circ (\Id + \sigma_{f^{-1}(x)}) \circ (A_{f^{-1}(x)}+\phi_{f^{-1}(x)})^{-1} \circ \Phi_x^{-1}(\tilde v) 
    \\
    & \qquad -  \Phi_{\tilde x} \circ \psi^s_{f^{-1}(\tilde x)} \circ (\Id + \sigma_{f^{-1}(\tilde x)}) \circ (A_{f^{-1}(\tilde x)}+\phi_{f^{-1}(\tilde  x)})^{-1} \circ \Phi_{\tilde x}^{-1}(\tilde v) \|,
\\
D_3&:= \| \Phi_x \circ \phi^s_{f^{-1}(x)}  \circ (A_{f^{-1}(x)}+\phi_{f^{-1}(x)})^{-1} \circ \Phi_x^{-1}(\tilde v) 
    \\
    & \qquad -  \Phi_{\tilde x} \circ \phi^s_{f^{-1}(\tilde x)}  \circ (A_{f^{-1}(\tilde x)}+\phi_{f^{-1}(\tilde  x)})^{-1} \circ \Phi_{\tilde x}^{-1}(\tilde v) \|.
\end{align*}
Now we estimate $D_1$, $D_2$ and $D_3$ separately. Here we only present details of the first estimate since others are similar.
By \eqref{Holder-A}, \eqref{Holder-phi},  remark \ref{rmk1}  and $f:X\to X$ being a Lipschitz homeomorphism,  we are able to calculate that
\begin{align*}
&D_1:=\| \Phi_x\circ A_{f^{-1}(x)}^s \circ \sigma_{f^{-1}(x)} \circ  (A_{f^{-1}(x)}+\phi_{f^{-1}(x)})^{-1} \circ \Phi_x^{-1}(\tilde v) 
\\
& \qquad-   \Phi_{\tilde x}\circ A_{f^{-1}(\tilde x)}^s \circ \sigma_{f^{-1}(\tilde x)} \circ(A_{f^{-1}(\tilde x)}+\phi_{f^{-1}(\tilde x)})^{-1} \circ \Phi_{\tilde x}^{-1}(\tilde v) \| 
\\
&\le \| \Phi_x\circ A_{f^{-1}(x)}^s \circ \hat \Phi^{-1}_{f^{-1}( x)}\circ \hat\Phi_{f^{-1}(x)} \circ \sigma_{f^{-1}(x)} \circ  (A_{f^{-1}(x)}+\phi_{f^{-1}(x)})^{-1} \circ \Phi_x^{-1}(\tilde v) 
\\
& \qquad - \Phi_{\tilde x}\circ A_{f^{-1}(\tilde x)}^s \circ \hat \Phi^{-1}_{f^{-1}(\tilde x)}\circ \hat \Phi_{f^{-1}(x)}\circ\sigma_{f^{-1}(x)} \circ  (A_{f^{-1}(x)}+\phi_{f^{-1}(x)})^{-1} \circ \Phi_x^{-1}(\tilde v) \|
\\
& \qquad + \|\Phi_{\tilde x}\circ A_{f^{-1}(\tilde x)}^s \circ \hat \Phi^{-1}_{f^{-1}(\tilde x)}\circ\hat \Phi_{f^{-1}(x)}\circ\sigma_{f^{-1}(x)} \circ  (A_{f^{-1}(x)}+\phi_{f^{-1}(x)})^{-1} \circ \Phi_x^{-1}(\tilde v)
\\
& \qquad - \Phi_{\tilde x}\circ A_{f^{-1}(\tilde x)}^s \circ\hat \Phi^{-1}_{f^{-1}(\tilde x)}\circ\hat \Phi_{f^{-1}(\tilde x)} \circ \sigma_{f^{-1}(\tilde x)} \circ(A_{f^{-1}(\tilde x)}+\phi_{f^{-1}(\tilde x)})^{-1} \circ \Phi_{\tilde x}^{-1}(\tilde v) \|
\\
& \le M_A\ell_f^\beta d(x,\tilde x)^\beta \| \hat \Phi_{f^{-1}(x)}\circ\sigma_{f^{-1}(x)} \circ  (A_{f^{-1}(x)}+\phi_{f^{-1}(x)})^{-1} \circ \Phi_x^{-1}(\tilde v)\| 
\\
& \qquad + \|\Phi_{\tilde x}\circ A_{f^{-1}(\tilde x)}^s \circ \hat\Phi^{-1}_{f^{-1}(\tilde x)}\| \cdot \|\hat \Phi_{f^{-1}(x)}\circ\sigma_{f^{-1}(x)} \circ  (A_{f^{-1}(x)}+\phi_{f^{-1}(x)})^{-1} \circ \Phi_x^{-1}(\tilde v)
\\
& \qquad - \hat\Phi_{f^{-1}(\tilde x)} \circ \sigma_{f^{-1}(\tilde x)} \circ(A_{f^{-1}(\tilde x)}+\phi_{f^{-1}(\tilde x)})^{-1} \circ \Phi_{\tilde x}^{-1}(\tilde v)\|
\\
& \le \frac{M_A \ell_f^\beta M}{(\ell-\delta)^\alpha}  d(x,\tilde x)^\beta \| \tilde v\|^\alpha
\\
& \qquad + \tau \| \hat\Phi_{f^{-1}(x)}\circ\sigma_{f^{-1}(x)} \circ \hat\Phi^{-1}_{f^{-1}(x)} \circ  \hat\Phi_{f^{-1}(x)} \circ (A_{f^{-1}(x)}+\phi_{f^{-1}(x)})^{-1} \circ \Phi_x^{-1}(\tilde v)
\\
& \qquad - \hat\Phi_{f^{-1}(\tilde x)}\circ\sigma_{f^{-1}(\tilde x)} \circ \hat\Phi^{-1}_{f^{-1}(\tilde x)} \circ  \hat\Phi_{f^{-1}(x)} \circ (A_{f^{-1}(x)}+\phi_{f^{-1}(x)})^{-1} \circ \Phi_x^{-1}(\tilde v) \|
\\
& \qquad+ \tau \| \hat\Phi_{f^{-1}(\tilde x)}\circ\sigma_{f^{-1}(\tilde x)} \circ \hat\Phi^{-1}_{f^{-1}(\tilde x)} \circ  \hat\Phi_{f^{-1}(x)} \circ (A_{f^{-1}(x)}+\phi_{f^{-1}(x)})^{-1} \circ \Phi_x^{-1}(\tilde v)
\\
& \qquad - \hat\Phi_{f^{-1}(\tilde x)} \circ \sigma_{f^{-1}(\tilde x)} \circ \hat\Phi^{-1}_{f^{-1}(\tilde x)} \circ  \hat\Phi_{f^{-1}(\tilde x)}\circ(A_{f^{-1}(\tilde x)}+\phi_{f^{-1}(\tilde x)})^{-1} \circ \Phi_{\tilde x}^{-1}(\tilde v)\|
\\
& \le \frac{M_A \ell_f^\beta M}{(\ell-\delta)^\alpha}  d(x,\tilde x)^\beta \| \tilde v\|^\alpha 
\\
&\qquad+\tau \tilde M d(f^{-1}(x),f^{-1}(\tilde x) )^{\tilde \alpha} \| \hat\Phi_{f^{-1}(x)} \circ (A_{f^{-1}(x)}+\phi_{f^{-1}(x)})^{-1} \circ \Phi_x^{-1}(\tilde v)\|^\alpha
\\
& \qquad + \tau M  \| \hat\Phi_{f^{-1}(x)} \circ (A_{f^{-1}(x)}+\phi_{f^{-1}(x)})^{-1} \circ \Phi_x^{-1}(\tilde v)
\\
& \qquad - \hat\Phi_{f^{-1}(\tilde x)}\circ(A_{f^{-1}(\tilde x)}+\phi_{f^{-1}(\tilde x)})^{-1} \circ \Phi_{\tilde x}^{-1}(\tilde v)\|^\alpha
\\
& \le \frac{M_A \ell_f^\beta M}{(\ell-\delta)^\alpha}  d(x,\tilde x)^\beta \| \tilde v\|^\alpha 
+ \frac{\tau\tilde M \ell_f^{\tilde \alpha}}{(\ell-\delta
 )^\alpha} 
d(x,\tilde x)^{\tilde \alpha} \| \tilde v\|^\alpha 
\\
 & \qquad  + (M_A^\alpha + M_\phi^\alpha)\frac{\tau C \ell_{f}^{\alpha\beta} M}{(\ell-\delta)^{2\alpha}} d(x,\tilde x)^{\alpha \beta } \| \tilde v\|^\alpha
 \\
 & \le \Bigg\{ \frac{M_A \ell_f^\beta M}{(\ell-\delta)^\alpha}  +  \frac{\tau\tilde M \ell_f^{\tilde \alpha}}{(\ell-\delta
 )^\alpha} 
+ (M_A^\alpha + M_\phi^\alpha)\frac{\tau C \ell_{f}^{\alpha\beta} M}{(\ell-\delta)^{2\alpha}} \Bigg\} d(x,\tilde x)^{\tilde \alpha} \| \tilde v\|^\alpha,
 \end{align*}
 where $\ell_f$ is the lower Lipschitz constant of $f$ and  $C$ is a constant sufficiently close to $1$, if $\tilde \alpha$ is chosen to satisfy $0<\tilde \alpha\le \alpha\beta$.
The last estimate in the above is given as follows:
\begin{align*}
&\| \tilde v-  \Phi_x \circ (A_{f^{-1(x)}}+\phi_{f^{-1}(x)}) \circ \hat\Phi^{-1}_{f^{-1}(x)} \circ \hat\Phi_{f^{-1}(\tilde x)}\circ(A_{f^{-1}(\tilde x)}+\phi_{f^{-1}(\tilde x)})^{-1} \circ \Phi_{\tilde x}^{-1}(\tilde v) \|^\alpha
\\
&=\| \Phi_x \circ (A_{f^{-1(x)}}+\phi_{f^{-1}(x)}) \circ \hat\Phi^{-1}_{f^{-1}(x)} \circ\hat\Phi_{f^{-1}(x)} \circ (A_{f^{-1}(x)}+\phi_{f^{-1}(x)})^{-1} \circ \Phi_x^{-1}(\tilde v)
\\
& \qquad - \Phi_x \circ (A_{f^{-1(x)}}+\phi_{f^{-1}(x)}) \circ \hat\Phi^{-1}_{f^{-1}(x)} \circ \hat \Phi_{f^{-1}(\tilde x)}\circ(A_{f^{-1}(\tilde x)}+\phi_{f^{-1}(\tilde x)})^{-1} \circ \Phi_{\tilde x}^{-1}(\tilde v)\|^\alpha    
\\
& \ge (\ell -\delta)^\alpha \|\hat\Phi_{f^{-1}(x)} \circ (A_{f^{-1}(x)}+\phi_{f^{-1}(x)})^{-1} \circ \Phi_x^{-1}(\tilde v)
\\
& \qquad - \hat \Phi_{f^{-1}(\tilde x)}\circ(A_{f^{-1}(\tilde x)}+\phi_{f^{-1}(\tilde x)})^{-1} \circ \Phi_{\tilde x}^{-1}(\tilde v)\|^\alpha.  
\end{align*}
On the other hand, letting $v:=\hat\Phi_{f^{-1}(\tilde x)}\circ(A_{f^{-1}(\tilde x)}+\phi_{f^{-1}(\tilde x)})^{-1} \circ \Phi_{\tilde x}^{-1}(\tilde v)$, the first line of the above estimate can be estimated as the following:
\begin{align*}
    &\text{The first line}
    \\
    & =\| \Phi_{\tilde x} \circ (A_{f^{-1}(\tilde x)}+\phi_{f^{-1}(\tilde x)}) \circ \hat\Phi_{f^{-1}(\tilde x)}^{-1}(v) - \Phi_x \circ (A_{f^{-1(x)}}+\phi_{f^{-1}(x)}) \circ \hat\Phi^{-1}_{f^{-1}(x)}(v)\|^\alpha
    \\
    & \le (M_A d(f^{-1}(x),f^{-1}(\tilde x))^\beta \| v\| )^\alpha+ (M_\phi d(f^{-1}(x),f^{-1}(\tilde x))^\beta \| v\| )^\alpha
    \\
    & \le (M_A^\alpha + M_\phi^\alpha)   \Big( \frac{ C\ell_{f}^\beta}{(\ell-\delta)} \Big)^\alpha  d(x,\tilde x)^{\alpha \beta } \| \tilde v\|^\alpha,
\end{align*}
where $C$ is a constant sufficiently close to $1$.
Then we get that 
\begin{align*}
& \|\hat\Phi_{f^{-1}(x)} \circ (A_{f^{-1}(x)}+\phi_{f^{-1}(x)})^{-1} \circ \Phi_x^{-1}(\tilde v)
 - \hat \Phi_{f^{-1}(\tilde x)}\circ(A_{f^{-1}(\tilde x)}+\phi_{f^{-1}(\tilde x)})^{-1} \circ \Phi_{\tilde x}^{-1}(\tilde v)\|^\alpha
 \\
& \le (M_A^\alpha + M_\phi^\alpha)\Big( \frac{C \ell_{f}^\beta}{(\ell-\delta)^2} \Big)^\alpha d(x,\tilde x)^{\alpha \beta } \| \tilde v\|^\alpha.
\end{align*}
For the estimate for $D_2$, we have that
\begin{align*}
    &D_2:=\| \Phi_x \circ \psi^s_{f^{-1}(x)} \circ (\Id + \sigma_{f^{-1}(x)}) \circ (A_{f^{-1}(x)}+\phi_{f^{-1}(x)})^{-1} \circ \Phi_x^{-1}(\tilde v) 
    \\
    & \qquad -  \Phi_{\tilde x} \circ \psi^s_{f^{-1}(\tilde x)} \circ (\Id + \sigma_{f^{-1}(\tilde x)}) \circ (A_{f^{-1}(\tilde x)}+\phi_{f^{-1}(\tilde  x)})^{-1} \circ \Phi_{\tilde x}^{-1}(\tilde v) \|
    \\
    & \le \| \Phi_x \circ \psi^s_{f^{-1}(x)} \circ \hat \Phi_{f^{-1}(x)}^{-1} \circ \hat \Phi_{f^{-1}(x)} \circ (\Id + \sigma_{f^{-1}(x)}) \circ (A_{f^{-1}(x)}+\phi_{f^{-1}(x)})^{-1} \circ \Phi_x^{-1}(\tilde v) 
    \\
    & \qquad - \Phi_{\tilde x} \circ \psi^s_{f^{-1}(\tilde x)} \circ \hat \Phi_{f^{-1}(\tilde x)}^{-1} \circ \hat \Phi_{f^{-1}(x)} \circ (\Id + \sigma_{f^{-1}(x)}) \circ (A_{f^{-1}(x)}+\phi_{f^{-1}(x)})^{-1} \circ \Phi_x^{-1}(\tilde v) \|
    \\
    & \qquad + \|\Phi_{\tilde x} \circ \psi^s_{f^{-1}(\tilde x)} \circ \hat \Phi_{f^{-1}(\tilde x)}^{-1} \circ \hat \Phi_{f^{-1}(x)} \circ (\Id + \sigma_{f^{-1}(x)}) \circ (A_{f^{-1}(x)}+\phi_{f^{-1}(x)})^{-1} \circ \Phi_x^{-1}(\tilde v) 
    \\
    & \qquad - \Phi_{\tilde x} \circ \psi^s_{f^{-1}(\tilde x)} \circ (\Id + \sigma_{f^{-1}(\tilde x)}) \circ (A_{f^{-1}(\tilde x)}+\phi_{f^{-1}(\tilde  x)})^{-1} \circ \Phi_{\tilde x}^{-1}(\tilde v) \|
    \\
   & \le \begin{cases}
        M_\psi \Big\{ \frac{1}{\ell-\delta} + \frac{M}{(\ell-\delta)^\alpha} \Big\} d(x,\tilde x)^\beta \| \tilde v\|^\alpha \qquad & \text{if } \|\tilde v\| <1,
        \\
        M_\psi^\alpha \Big\{ 
        \frac{1}{(\ell-\delta)^\alpha}  + \Big(\frac{M}{(\ell-\delta)^\alpha} \Big)^\alpha \Big\} d(x,\tilde x)^{\alpha\beta}  \| \tilde v\|^\alpha \qquad& \text{if } \|\tilde v\| \ge 1,
    \end{cases}
    \\
    & \qquad + \delta \|\hat \Phi_{f^{-1}(x)} \circ (\Id + \sigma_{f^{-1}(x)}) \circ (A_{f^{-1}(x)}+\phi_{f^{-1}(x)})^{-1} \circ \Phi_x^{-1}(\tilde v) 
    \\
    & \qquad -\hat\Phi_{f^{-1}(\tilde x)} \circ  (\Id + \sigma_{f^{-1}(\tilde x)}) \circ (A_{f^{-1}(\tilde x)}+\phi_{f^{-1}(\tilde  x)})^{-1} \circ \Phi_{\tilde x}^{-1}(\tilde v) \|
    \\
    & \le \tilde C d(x,\tilde x)^{\tilde \alpha} \| \tilde v\|^\alpha +  \delta(M_A^\alpha + M_\phi^\alpha)\Big( \frac{C \ell_{f}^\beta}{(\ell-\delta)^2} \Big)^\alpha d(x,\tilde x)^{\alpha \beta } \| \tilde v\|^\alpha
    \\
    & \qquad +\frac{\delta\tilde M \ell_f^{\tilde \alpha}}{(\ell-\delta
 )^\alpha} 
d(x,\tilde x)^{\tilde \alpha} \| \tilde v\|^\alpha 
 + \delta(M_A^\alpha + M_\phi^\alpha)\frac{ C \ell_{f}^{\alpha\beta} M}{(\ell-\delta)^{2\alpha}} d(x,\tilde x)^{\alpha \beta } \| \tilde v\|^\alpha,
\end{align*}
where 
\begin{align}
  \tilde C&:=\max \Bigg\{       M_\psi \Big\{ \frac{1}{\ell-\delta} + \frac{M}{(\ell-\delta)^\alpha} \Big\}, 
  M_\psi^\alpha \Big\{ 
        \frac{1}{(\ell-\delta)^\alpha}  + \Big(\frac{M}{(\ell-\delta)^\alpha} \Big)^\alpha \Big\} \Bigg\}.   
        \label{tilde-C}
\end{align}
Here we need $B_\phi$ defined by \eqref{bound-perturb.} to be small.
The third estimate for $D_3$ is similar. Thus, for each $\sigma\in \mathcal C_b^{\alpha,\tilde \alpha}(E) $, by \eqref{operator},  \eqref{Holder-A}, \eqref{Holder-phi},  remark \ref{rmk1} and $f:X\to X$ being a Lipschitz homeomorphism, we can calculate that 
\begin{align}
    &\| \Phi \circ P_x\circ(\mathcal T \sigma) \circ \Phi^{-1} (x,\tilde v) - \Phi \circ P_{\tilde x}\circ(\mathcal T \sigma) \circ \Phi^{-1} (\tilde x,\tilde v) \| \notag
    \\
    & \le  D_1 + D_2 +D_3 \notag
    \\
    & \le \Bigg\{ \frac{M_A \ell_f^\beta M}{(\ell-\delta)^\alpha}  +  \frac{\tau\tilde M \ell_f^{\tilde \alpha}}{(\ell-\delta
 )^\alpha} 
+ (M_A^\alpha + M_\phi^\alpha)\frac{\tau C \ell_{f}^{\alpha\beta} M}{(\ell-\delta)^{2\alpha}} \notag
\\
& \qquad + \tilde C +   \delta(M_A^\alpha + M_\phi^\alpha)\Big( \frac{C \ell_{f}^\beta}{(\ell-\delta)^2} \Big)^\alpha (1+M) +\frac{\delta\tilde M \ell_f^{\tilde \alpha}}{(\ell-\delta
 )^\alpha} \notag
\\
&\qquad +\hat C + \Big(\frac{\delta C^2 \ell_f^\beta(M_A+M_\phi)}{(\ell-\delta)^2} \Big)^\alpha \Bigg\} d(x,\tilde x)^{\tilde \alpha} \| \tilde v\|^\alpha \notag
\\
& =: \Big\{\tilde M (\tau+\delta) \frac{\ell_f^{\tilde \alpha}}{(\ell-\delta
 )^\alpha} + \bar C \Big\} d(x,\tilde x)^{\tilde \alpha} \| \tilde v\|^\alpha
 \label{estim.-base1}
\end{align}
where $\tilde C$ is given in \eqref{tilde-C}, $C$ is a constant close to $1$,
\[ \hat C := \max \Bigg\{ \frac{M_\phi \ell_f^\beta}{\ell-\delta}, \bigg( \frac{M_\phi \ell_f^\beta}{\ell-\delta}\bigg)^\alpha \Bigg\} \]
and 
\begin{align*}
    \bar C&:= \frac{M_A \ell_f^\beta M}{(\ell-\delta)^\alpha}  
+ (M_A^\alpha + M_\phi^\alpha)\frac{\tau C \ell_{f}^{\alpha\beta} M}{(\ell-\delta)^{2\alpha}} + \tilde C +   \delta(M_A^\alpha + M_\phi^\alpha)\Big( \frac{C \ell_{f}^\beta}{(\ell-\delta)^2} \Big)^\alpha (1+M) 
\\
& \qquad +\hat C + \Big(\frac{\delta C^2 \ell_f^\beta(M_A+M_\phi)}{(\ell-\delta)^2} \Big)^\alpha .
\end{align*}

Analogously, we can have 
\begin{align}
    &\| \Phi \circ Q_x\circ(\mathcal T \sigma) \circ \Phi^{-1} (x,\tilde v) - \Phi \circ Q_{\tilde x}\circ(\mathcal T \sigma) \circ \Phi^{-1} (\tilde x,\tilde v) \| \notag
    \\
    & \le M_AM d(x,\tilde x)^\beta \| \tilde v\|^\alpha  + \tau C M_\phi^\alpha d(x,\tilde x)^{\alpha\beta} \| \tilde v\|^\alpha + M_A M (L+\delta)^\alpha d(x,\tilde x)^{\beta} \| \tilde v\|^\alpha  \notag
    \\
    &  \qquad
    +\tau C \tilde M (L+\delta)^\alpha {(L_f)}^{\tilde \alpha} d(x,\tilde x)^{\tilde \alpha} \| \tilde v\|^\alpha
     + \tau C M \{(M_A)^\alpha + (M_\phi)^\alpha\} d(x,\tilde x)^{\alpha\beta} \| \tilde v\|^\alpha \notag
     \\
    &\qquad + M_A \max \{ \delta(1+M), \delta^\alpha(1+M^\alpha)\} d(x,\tilde x)^\beta \|\tilde v\|^\alpha  + C^2 \tau \delta \tilde M d(x,\tilde x)^{\tilde \alpha} \|\tilde v\|^\alpha \notag
    \\
    & \qquad + C\tau \max \{ M_\phi(1+M), M_\phi(1+M^\alpha)\} d(x,\tilde x)^{\tilde \alpha}\|\tilde v\|^\alpha
    \notag
    \\
    & =:\{  \tilde M \tau ( C(L+\delta)^\alpha(L_f)^{\tilde \alpha} + C^2 \delta)  + \stackrel{=}{C}   \} 
    d(x,\tilde x)^{\tilde \alpha} \| \tilde v\|^\alpha,
    \label{estim.-base2}
\end{align}
where 
\begin{align*}
  \stackrel{=}{C} &:= M_A M +    \tau C M_\phi^\alpha +  M_A M (L+\delta)^\alpha
  \\
  & \qquad +\tau C M \{(M_A)^\alpha + (M_\phi)^\alpha\} + M_A \max \{ \delta(1+M), \delta^\alpha(1+M^\alpha)\} 
  \\
  & \qquad + C\tau \max \{ M_\phi(1+M), M_\phi(1+M^\alpha)\}.
\end{align*}

Combining \eqref{estim.-base1} with \eqref{estim.-base2}, we obtain that 
\begin{align*}
&\| \Phi \circ (\mathcal T \sigma) \circ \Phi^{-1} (x,\tilde v) - \Phi \circ (\mathcal T \sigma) \circ \Phi^{-1} (\tilde x,\tilde v) \|
\\
&\le \max \{  \| \Phi \circ P_x\circ(\mathcal T \sigma) \circ \Phi^{-1} (x,\tilde v) - \Phi \circ P_{\tilde x}\circ(\mathcal T \sigma) \circ \Phi^{-1} (\tilde x,\tilde v) \|,
\\
&\qquad \| \Phi \circ Q_x\circ(\mathcal T \sigma) \circ \Phi^{-1} (x,\tilde v) - \Phi \circ Q_{\tilde x}\circ(\mathcal T \sigma) \circ \Phi^{-1} (\tilde x,\tilde v) \| \}
\\
&\le \max  \Bigg\{  \tilde M (\tau+\delta) \frac{\ell_f^{\tilde \alpha}}{(\ell-\delta
 )^\alpha} + \bar C ,
 \tilde M \tau ( C(L+\delta)^\alpha(L_f)^{\tilde \alpha} + C^2 \delta)  + \stackrel{=}{C} 
\Bigg\}
    d(x,\tilde x)^{\tilde \alpha} \| \tilde v\|^\alpha,
\end{align*}
from which we see that 
\begin{equation*}
    \| \Phi \circ (\mathcal T \sigma) \circ \Phi^{-1} (x,\tilde v) - \Phi \circ (\mathcal T \sigma) \circ \Phi^{-1} (\tilde x,\tilde v) \| \le \tilde M  d(x,\tilde x)^{\tilde \alpha} \| \tilde v\|^\alpha
\end{equation*}
if we take 
\begin{align*}
    \tilde M \ge \max \Bigg\{  \frac{\bar C}{1- (\tau+\delta) \frac{\ell_f^{\tilde \alpha}}{(\ell-\delta
 )^\alpha} }, \frac{\stackrel{=}{C}}{1-\tau ( C(L+\delta)^\alpha(L_f)^{\tilde \alpha} + C^2 \delta) }
 \Bigg\}
\end{align*}
because we can choose positive $\tilde \alpha$ and $\alpha$ such that 
\begin{equation}\label{less-one}
\max\Big\{ (\tau+\delta) \frac{\ell_f^{\tilde \alpha}}{(\ell-\delta
 )^\alpha}, \tau ( C(L+\delta)^\alpha(L_f)^{\tilde \alpha} + C^2 \delta) \Big\}   < 1
\end{equation}
due to the condition $\delta<1-\tau$.
Consequently, $\mathcal T$ is a self-mapping on the set $C_b^{\alpha,\tilde \alpha}(E)$ for some positive $\tilde \alpha$ and $\alpha$. In other words, the solution obtained in Theorem \ref{thm1} is a H\"older continuous along the base space  $X$ for some positive exponent $\tilde \alpha$. Similar arguments can be applied to the function $\tilde h$. The proof is completed.
\end{proof}

As a special case, Theorem \ref{thm2} is immediately obtained by Lemma \ref{lm-Holder-along.base}. 

\section{Proof of Theorem \ref{thm3}}

First, we need a cut-off function defined on $E$. Consider a finite cover $(U_i)_{i=1}^n$ of $X$, and continuous local trivializations 
$\Phi_i : \pi^{-1}(U_i) \to U_i \times \mathbb R^d$. Choose a smooth cut-off function $\rho:\mathbb R^d \to \mathbb R$ such that 
\begin{equation*}
 \rho(t) = \begin{cases}
     1 & \text{for } \|t\| \le r,
     \\
     0 &  \text{for } \|t\| \ge 2r,
     \\
     \in (0,1) & \text{for } r <\|t\| < 2r,
 \end{cases}   
\end{equation*}
where $r>0$ is a small constant. Now, we define a local cut-off function $\rho_i$ on $\pi^{-1}(U_i)$ for each $i$ to be
\[\rho_i(v) = \rho(\|\Phi_i|_x(v)\|) \qquad \text{for each }  p \in U_i \text{ and each }  v\in E_x. \]
Using partition of unity $(\eta_i)_{i=1}^n$ subordinate to $(U_i)_{i=1}^n$, we get a cut-off function on $E$ defined by
\begin{equation}\label{cut-off}
    \xi(v) = \sum_{i=1}^n \eta_i(x) \cdot \rho_i(v)
\end{equation}
where $x$ is the base point of $v$, $0\le \eta_i(x)\le  1$ for each $1\le i\le n$ and each $x\in X$, supp $\eta_i \subset U_i$ for each $1\le i\le n$ and 
\[ \sum_{i=1}^n \eta_i(x) = 1 \qquad \text{ for all } x\in X.\]

Write 
\[ F= A + \phi,\]
where $A_x:E_x \to E_{f(x)}$ is the derivative of $F_x$ at the origin of $E_x$ for each $x\in X$. Then $\phi$ is a $C^1$ map from $E$ into itself satisfying $\phi(O_x)= O_{f(x)}$ and $D\phi_x(O_x) =0$ for each $x\in X$, where $O_x$  stands for the origin of $E_x$.  
Let 
\[ \tilde F :=A+\tilde \phi:=A+\xi \cdot \phi.\]
Obviously, $\tilde F=F$ at a neighborhood of zero section of $E$. Now we check that $A$ and $\tilde \phi$ satisfy the conditions of Theorems 1 and 2. Note that $A$ is assumed to be hyperbolic.
For each $x\in X$, by mean value theorem we see that 
\[ \|\tilde \phi_x(v_1) - \tilde \phi_x(v_2)\| \le \| D\tilde \phi_x(\xi)\| \cdot \| v_1 -v_2\|.  \]
Since $D\tilde \phi_x(O_x)=0$, we can choose $r>0$ so small so that  $\| D\tilde \phi_x(\xi)\| \le \min\{ 1-\tau,\ell\}$ and hence 
\[ \|\tilde \phi_x(v_1) - \tilde \phi_x(v_2)\| \le \min\{ 1-\tau,\ell\}  \cdot \| v_1 -v_2\|.  \]
Moreover, we have 
\[ \sup_{x\in X} \sup_{v\in E_x} \| \tilde \phi_x (v)\| < \infty  \]
to be sufficiently small since $X$ is compact, $\tilde \phi =0$ outside a neighborhood of zero section and $\| \cdot\|$ is continuous by making $r$ sufficiently small. The conditions of Theorem \ref{thm1} are satisfied.

 For $x,\tilde x\in X$ to be close, 
\begin{align*}
    &\| \hat\Phi_{f(x)} \circ A_x \circ \Phi_x^{-1} - \hat\Phi_{f(\tilde x)} \circ A_{\tilde x} \circ \Phi_{\tilde x}^{-1}\| 
    \\
    & = \| \hat\Phi_{f(x)} \circ DF_x(O_x) \circ \Phi_x^{-1} - \hat\Phi_{f(\tilde x)} \circ DF_{\tilde x}(O_{\tilde x}) \circ \Phi_{\tilde x}^{-1}\| 
    \\
    & \le \| \frac{\partial \hat \Phi \circ F \circ \Phi^{-1}}{\partial x \partial {\tilde v}} (\theta, 0)\| \cdot \| \varphi(x) - \varphi(\tilde x)\|
    \\
    & \le \| \frac{\partial \hat \Phi \circ F \circ \Phi^{-1}}{\partial x \partial {\tilde v}} (\theta, 0)\| \cdot L_\varphi d(x,\tilde x)
    \\
    & =: M_A d(x,\tilde x) 
    \\
    & \le M_A d(x,\tilde x)^\beta,
\end{align*}
where $(U,\varphi)$ is a smooth chart of $X$ containing $x$ and $\tilde x$, $\theta$ is a point on the segment connecting $\varphi(x)$ to $\varphi(\tilde x)$, and $0<\beta<1$. It is easy to check that $M_A<\infty$ by compactness of $X$.
In addition, 
\begin{align*}
   & \|\hat\Phi_{f(x)} \circ \tilde \phi_x \circ \Phi_x^{-1}(\tilde v) - \hat \Phi_{f(\tilde x)} \circ \tilde \phi_{\tilde x} \circ \Phi_{\tilde x}^{-1}(\tilde v) \| 
   \\
   & \le \|\hat\Phi_{f(x)} \circ (\tilde F_x - A_x) \circ \Phi_x^{-1}(\tilde v) - \hat \Phi_{f(\tilde x)} \circ (\tilde F_{\tilde x} - A_{\tilde x}) \circ \Phi_{\tilde x}^{-1}(\tilde v) \| 
   \\
   & \le \| \frac{\partial \hat \Phi \circ \tilde F \circ \Phi^{-1} }{\partial x}(\theta_1, \tilde v)\| d(x,\tilde x) + M_A d(x,\tilde x)^\beta \| \tilde v\|
   \\
   & = \| \frac{\partial \hat \Phi \circ \tilde F \circ \Phi^{-1} }{\partial x}(\theta_1, \tilde v) -\frac{\partial \hat \Phi \circ \tilde F \circ \Phi^{-1} }{\partial x}(\theta_1, 0) \| d(x,\tilde x) + M_A d(x,\tilde x)^\beta \| \tilde v\|
   \\
   &\le \| \frac{\partial^2 \hat \Phi \circ \tilde F \circ \Phi^{-1} }{\partial x\partial \tilde v}(\theta_1, \theta_2) \| \| \tilde v\|  d(x,\tilde x) + M_A d(x,\tilde x)^\beta \| \tilde v\| 
\\   
& \le  M_{\tilde \phi} d(x,\tilde x)^\beta \| \tilde v\|, 
\end{align*}
where
\[ M_{\tilde \phi} :=  \sup_{x\in X, v\in E_x}\left\| \frac{\partial^2  \tilde F  }{\partial x\partial  v}(x, v) \right\| +M_A <\infty\]
because $\tilde F$ is zero outside of a neighborhood of zero section.  We can also choose $r>0$ sufficiently small such that $B_{\tilde \phi}$ defined by \eqref{bound-perturb.} is also sufficiently small. Thus, we can apply Theorem \ref{thm1} and \ref{thm2} to the function $\tilde F$. Since $\tilde F=F$ near the zero section, Theorem \ref{thm3} is proved.  

{\bf Acknowledgements:}
This work is supported by  NSF-CQ \#CSTB2023NSCQ-JQX0020, NSFC
\#12271070 and the Science and Technology Research Program of Chongqing Municipal Education Commission (Grant No. KJQN202300540).

\appendix






\end{document}